\documentclass[a4paper, 11pt,reqno]{amsart}
\usepackage[normalem]{ulem}

\usepackage[T1]{fontenc}      % Uses 'T1' font encoding.
\usepackage[foot]{amsaddr}    % Manages options for displayed author information.

\usepackage[british]{babel}   % Support for different languages.
\usepackage{csquotes}         % Adds quotation marks according to language.
\usepackage[babel]{microtype} % Makes things look nicer (spacing and so on).

\usepackage{mathtools}        % Extension of 'amsmath' which fixes some bugs.
\usepackage{amssymb}          % Extra symbols and fonts.
\usepackage{amsthm}           % Helps to define "theorem-like" environments.
\usepackage{amstext}          % Defines a '\text' macro for math environments.
\usepackage{mleftright}       % Fixes spacing issue.
\usepackage{gensymb}          % Allows to use degrees.
\usepackage[makeroom]{cancel} % Defines cancellation of terms in formulae.
\usepackage{mathdots}         % Adds some definitions of dots.
\usepackage{mathrsfs}         % Adds rsfs fonts.
\usepackage{fancyvrb}
\usepackage{stickstootext}
\usepackage[stickstoo,varbb]{newtxmath} % These two lines are an alternative for stix2 which do essentially the same but fix several issues with some symbols...
\makeatletter
\DeclareFontFamily{U}{ntxmia}{\skewchar \font =127}
 \DeclareFontShape{U}{ntxmia}{m}{it}{
                        <-> \ntxmath@scaled ntxmia
                      }{}    
                      \DeclareFontShape{U}{ntxmia}{b}{it}{
                        <-> \ntxmath@scaled ntxbmia
                      }{}
\makeatother

\usepackage{graphicx}         % Provides commands to include graphics.
\usepackage{psfrag}           % Gives some extra useful things for graphics.
\usepackage{caption}          % Options for formatting float captions.
\usepackage{tikz}             % Language to create graphics programmatically.
\usetikzlibrary{backgrounds}

\usepackage{booktabs}         % Enhances the quality of tables.

\usepackage{enumitem}         % Allows customization for enumerations.

\usepackage[margin=1in]{geometry} % Provides commands to define the page layout.

\usepackage[textsize=footnotesize]{todonotes}
\usepackage[numbers,sort]{natbib}

\makeatletter
\def\NAT@spacechar{~}% NEW
\makeatother

\usepackage{hyperref}              % For hyperlinks within the document
\usepackage[capitalise]{cleveref}  % For improved referencing options; use [capitalise] if desired
\hypersetup{colorlinks=true,
   citecolor=blue,
   filecolor=blue,
   linkcolor=blue,
   urlcolor=blue
}
\usepackage{url}

\makeatletter
\if@cref@capitalise
\crefname{figure}{figure}{figures}
\crefname{claim}{Claim}{Claims}
\else
\crefname{figure}{Figure}{Figures}
\crefname{claim}{claim}{claims}
\fi
\makeatother
\Crefname{figure}{Figure}{Figures}
\Crefname{claim}{Claim}{Claims}

\usepackage{algorithm}
\usepackage{algpseudocode}

\allowdisplaybreaks

\theoremstyle{definition}
\newtheorem{definition}{Definition}
\newtheorem{remark}[definition]{Remark}

\theoremstyle{plain}
\newtheorem{claim}{Claim}

\newtheorem{theorem}[definition]{Theorem}
\newtheorem{corollary}[definition]{Corollary}
\newtheorem{lemma}[definition]{Lemma}

\newtheorem{conjecture}[definition]{Conjecture}

\newtheorem{problem}[definition]{Problem}

\newenvironment{claimproof}{%
\let\origqed=\qedsymbol%
\renewcommand{\qedsymbol}{$\blacktriangleleft$}%
\begin{proof}}{\end{proof}\let\qedsymbol=\origqed}

\numberwithin{equation}{section}

\renewcommand{\binom}[2]{\ensuremath{\mleft(\kern-.1em\genfrac{}{}{0pt}{}{#1}{#2}\kern-.1em\mright)}}    % This makes binomial numbers nicer with stix2 (in displayed equations). Remove if stix2 is not loaded.
\newcommand{\inbinom}[2]{\ensuremath{\bigl(\kern-.1em\genfrac{}{}{0pt}{}{#1}{#2}\kern-.1em\bigr)}} % This is better for inline equations, as it will keep sizes of parentheses consistent and not create extra vertical space.

\newcommand*\nume{\ensuremath{\mathrm{e}}}

\newcommand{\cA}{\mathcal{A}}

\newcommand{\cE}{\mathcal{E}}
\newcommand{\cF}{\mathcal{F}}
\newcommand{\cG}{\mathcal{G}}

\newcommand{\cS}{\mathcal{S}}

\newcommand{\cW}{\mathcal{W}}
\newcommand{\cX}{\mathcal{X}}

\newcommand{\PP}{\mathbb{P}}
\renewcommand{\Pr}{\PP}

\makeatletter
\def\moverlay{\mathpalette\mov@rlay}
\def\mov@rlay#1#2{\leavevmode\vtop{%
  \baselineskip\z@skip \lineskiplimit-\maxdimen
  \ialign{\hfil$\m@th#1##$\hfil\cr#2\crcr}}}
\newcommand{\charfusion}[3][\mathord]{
    #1{\ifx#1\mathop\vphantom{#2}\fi
        \mathpalette\mov@rlay{#2\cr#3}
      }
    \ifx#1\mathop\expandafter\displaylimits\fi}
\makeatother

\newcommand{\dist}{\operatorname{dist}}

\newcommand{\eps}{\epsilon}

\newcommand{\COMMENT}[1]{}
\newcommand{\COMNEW}[1]{}
\renewcommand{\COMNEW}[1]{\footnote{\textcolor{red!70!black}{#1}}} % comment out to hide comments

\title[Sharp threshold for embedding balanced spanning trees in random geometric graphs]{Sharp threshold for embedding balanced spanning trees\linebreak{} in random geometric graphs}

\author[A.~Espuny D\'iaz]{Alberto Espuny D\'iaz}
\email{alberto.espuny-diaz@tu-ilmenau.de}
\address[Espuny D\'iaz]{Institut f\"ur Mathematik, Technische Universit\"at Ilmenau, 98684 Ilmenau, Germany.}
\author[L.~Lichev]{Lyuben Lichev}
\email{lyuben.lichev@univ-st-etienne.fr}
\address[Lichev]{Institut Camille Jordan, Univ. Jean Monnet, Saint-Etienne, France}
\author[D.~Mitsche]{Dieter Mitsche}
\email{dieter.mitsche@mat.uc.cl}
\address[Mitsche]{IMC, Pont.\ Univ.\ Cat\'{o}lica, Chile and Institut Camille Jordan, Univ.\ Jean Monnet, Saint-Etienne, France}
\author[A.~Wesolek]{Alexandra Wesolek}
\email{agwesole@sfu.ca}
\address[Wesolek]{Department of Mathematics, Simon Fraser University, Burnaby, BC, Canada}

\thanks{The research leading to these results has been supported by the Carl-Zeiss-Foundation and by DFG grant PE 2299/3-1 (A.~Espuny D\'iaz), by grant GrHyDy ANR-20-CE40-0002 and by Fondecyt grant 1220174 (D.~Mitsche) and by the Vanier Scholarship Program (A.~Wesolek). }

\date{\today}

\begin{document}

\begin{abstract}
A rooted tree is \emph{balanced} if the degree of a vertex depends only on its distance to the root.
In this paper we determine the sharp threshold for the appearance of a large family of balanced spanning trees in the random geometric graph $\cG(n,r,d)$.
In particular, we find the sharp threshold for balanced binary trees.
More generally, we show that \emph{all} sequences of balanced trees with uniformly bounded degrees and height tending to infinity appear above a sharp threshold, and none of these appears below the same value. 
Our results hold more generally for geometric graphs satisfying a mild condition on the distribution of their vertex set, and we provide a polynomial time algorithm to find such trees.
\end{abstract}

% \vspace{1em}
% Keywords: .
% \vspace{1em}
% MSC: .

\maketitle

\section{Introduction}

The \emph{random geometric graph} $\cG(n,r,d)$ is a classic model of random graphs defined as follows. Let $d$ and $n$ be positive integers, and let $r$ be a positive real number.
The vertices of the graph are $n$ points sampled uniformly at random and independently from $[0,1]^d$, and two vertices are connected by an edge if their Euclidean distance is at most $r$.
Since their introduction by Gilbert~\cite{Gil61} as a model for telecommunication networks, random geometric graphs have received a lot of attention both from an applied point of view~\cite{Aky02, Nek07, FLMPSSSvL19} and from a theoretical point of view~\cite{Pen03, Pen16}.
Moreover, the original model has been generalized in many different ways; for example, \citet{Wax88} introduced a model with additional percolation of the edges, which has been further studied in~\cite{DG16, LLMS22, Pen16a}.
Here, we focus on the classic model defined above. 

One main reason for the substantial interest in random geometric graphs is their use as a model for wireless networks.
Such networks consist of a set of nodes, each of them equipped with a wireless transceiver to communicate with their nearest neighbors (in terms of Euclidean distance).
The ability of communication (controlled by the range $r$) depends on the transmitting power of the transceivers.
The goal is to spread information through the network, which is done in a multi-hop fashion.
In many ad hoc networks, like sensor networks, energy consumption is an issue.
Therefore, one of the most important questions when modeling a network is how to minimize power consumption (see, for example, \cite{Aky02, AB02, ZG03}).
To do this, the transmission range should be made as small as possible, but at the same time large enough to make sure that a piece of information transmitted from a node will arrive to all other nodes in the network. Spanning trees are especially interesting from this application point of view since they are minimally connected sets.
In particular, when restricting these trees to having a particular structure, this raises the question of which transmission range is needed in order to ensure such trees in a graph.

Random geometric graphs are known to exhibit \emph{threshold} behavior for many graph properties, meaning that there are some special values of the parameters of the model around which a drastic change in the behavior of the graph with respect to these properties takes place.
Understanding these thresholds is one of the main directions of research in the theory of random graphs.
In the current work, we show that such a phenomenon takes place for the property of containing a wide range of ``sufficiently symmetric'' spanning trees.

\subsection{Thresholds in random geometric graphs}

Formally, a \emph{graph property} (or just property) is a set of labeled graphs which is closed under isomorphism.
A property is said to be \emph{monotone increasing} (resp. monotone decreasing) if it is preserved under edge addition (resp. edge deletion).
In particular, when thinking of $\cG(n,r,d)$, an increasing (resp.\ decreasing) property is preserved after increasing (resp.\ decreasing) the radius.
This leads to the definition of thresholds in random geometric graphs: a function $r^*=r^*(n,d)$ is a \emph{threshold} for some monotone increasing property $\mathcal{P}$ in $\cG(n,r,d)$ if
\[\lim_{n\to\infty}\mathbb{P}[\cG(n,r,d)\in\mathcal{P}]=
\begin{cases}
0&\text{if }r=o(r^*),\\
1&\text{if }r=\omega(r^*).
\end{cases}\]
Moreover, we say that $r^*$ is a \emph{sharp threshold} for $\mathcal{P}$ if, for every $\epsilon\in (0,1)$,
\[\lim_{n\to\infty}\mathbb{P}[\cG(n,r,d)\in\mathcal{P}]=
\begin{cases}
0&\text{if }r \le (1-\epsilon)r^*,\\
1&\text{if }r \ge (1+\epsilon)r^*.
\end{cases}\]
In this context, \citet{McC04} proved that all monotone increasing properties have a threshold in $\cG(n,r,1)$, and also that any such property whose threshold is much larger than $\log n/n$ must have a sharp threshold (throughout the paper, $\log$ stands for the natural logarithm).
\citet{GRK05} gave general upper bounds for the threshold width in $\cG(n,r,d)$ for all $d\geq 1$, and \citet{BP14} characterized vertex-monotone properties which exhibit a sharp threshold.
While the results in~\cite{McC04,GRK05,BP14} serve to prove the existence of (sharp) thresholds, they give no indication of where these thresholds actually are.
Determining the (sharp) thresholds for different properties of interest is one of the main problems in the area, and it has received much attention.

\subsection{Spanning trees in random geometric graphs}

A \emph{spanning} subgraph of a graph $G$ is a subgraph whose vertex set is $V(G)$.
In this paper, we are interested in determining the (sharp) threshold for the appearance of different spanning trees in random geometric graphs. 

A necessary condition for the containment of any spanning tree is that the graph must be connected.
The sharp threshold for connectivity in $\mathcal{G}(n,r,1)$ was determined by \citet{GJ96}, who showed that it is $\log n/n$.
In $\mathcal{G}(n,r,2)$, \citet{GuptaKumar} and \citet{Penrose97} independently showed that the sharp threshold for connectivity is $\sqrt{\log n/\uppi n}$.
Later, \citet[Theorem~13.2]{Pen03} showed that, for all $d\geq2$, the sharp threshold for connectivity in $\cG(n,r,d)$ is $(2^{d-1}\log n/(d\theta_dn))^{1/d}$, where $\theta_d$ is the volume of a unit ball in $\mathbb{R}^d$.
In fact, he even showed that, as $r$ increases, typically 
the graph becomes connected exactly when its last isolated vertex disappears (see~\cite[Theorem~13.17]{Pen03}). 
These results are crucial towards understanding the properties of the \emph{minimum spanning tree} of $\cG(n,r,d)$. 
However, they do not give us any information about the threshold for the appearance of \emph{specific} spanning trees.

An important special case of the problem of finding a fixed spanning tree is the containment of a spanning path.
On the one hand, D\'iaz, Mitsche and P\'erez-Gim\'enez~\cite{DMP07} showed that the sharp threshold for $\mathcal{G}(n,r,2)$ to contain a Hamilton cycle (i.e., a spanning cycle) is $\sqrt{\log n/\uppi n}$, that is, the same as the sharp threshold for connectivity. 
This result was later extended by \citet{BBKMW11} and \citet{MPW11}, who showed that the sharp threshold for Hamiltonicity in $\cG(n,r,d)$ coincides with the sharp threshold for connectivity for all $d\geq2$.
In fact, they showed that typically a Hamilton cycle appears in $\cG(n,r,d)$ once all vertices have degree at least~$2$. 
On the other hand, when $d=1$, scanning the points from left to right easily shows that the graph contains a spanning path as long as it is connected.
In particular, these results imply that the sharp threshold for containing a spanning path in $\cG(n,r,d)$ is the same as the sharp threshold for connectivity for all $d\geq1$.

The question of trying to find the (sharp) threshold for the appearance of different families of spanning trees pops up naturally.
In the case of ``path-like'' trees, one may obtain the threshold directly from the threshold for Hamiltonicity. 
Indeed, by the triangle inequality, if $\cG(n,r,d)$ contains a Hamilton cycle and $k\in\mathbb{N}$, then $\cG(n,kr,d)$ contains the $k$-th power of this Hamilton cycle.
It immediately follows that every spanning tree which can be embedded into the $k$-th power of a Hamilton cycle has the same threshold as Hamiltonicity.
This is the case, for instance, of spanning caterpillars with constant maximum degree.

One may naturally wonder whether all spanning trees with bounded maximum degree have the same threshold.
Incidentally, in the model of binomial random graphs $\mathcal{G}(n,p)$ where each of the $\inbinom{n}{2}$ possible edges appears independently with probability $p$, \citet{Montgomery19} proved that this is the case: the threshold for all bounded-degree spanning trees is $\log n/n$.
However, this turns out to be very far from the truth in random geometric graphs.
Indeed, there are bounded-degree trees $T$ whose diameter is logarithmic in the number of vertices, and this diameter directly imposes a much higher lower bound on the threshold $r^*$ for the property of containing a copy of $T$:
since a spanning subgraph of $\cG(n,r,d)$ cannot have smaller diameter than $\cG(n,r,d)$ itself, the threshold for spanning trees of diameter $O(\log n)$ must satisfy $r^*=\Omega(1/\log n)$.
This is far larger than the connectivity threshold mentioned above.
The results of \citet{GRK05} imply that, for any such tree, there is a sharp threshold.
Our goal is to determine the value of this threshold.

Note that, among all trees with logarithmic diameter, binary trees are especially interesting due to their many applications as data structures (see, e.g.,~\cite{IntroAlg}).
Identifying the sharp threshold for embedding these trees in $\cG(n,r,d)$ is thus an important particular case of our study.

\subsection{Main results}

As mentioned above, our focus is on determining the sharp threshold for the appearance of certain trees having thresholds at large radial values.
For a graph $G$, we denote by $|G|$ the size of the vertex set of $G$.

The trees $T$ we consider will have a special vertex which we will call the \emph{root}.
We may think of the vertices of $T$ as being partitioned into \emph{layers} $V_0,V_1,\ldots$, where $V_i$ contains all vertices at (graph) distance $i$ from the root.
The \emph{height} of $T$ is the maximum $h\in\mathbb{N}$ such that $V_h$ is empty.
For any vertex $v\in V(T)$, if $v\in V_j$, we refer to its neighbors in $V_{j+1}$ as its \emph{children}, and to all vertices which can be reached by a path from $v$ without going through $V_{j-1}$ as its \emph{descendants}.
Given a positive integer $s$, we say that a tree is an \emph{$s$-ary tree} if all its vertices have at most $s$ children.
We say that an $s$-ary tree of height $h$ is \emph{complete balanced} if all vertices in $V_0,V_1,\ldots,V_{h-1}$ have $s$ children.
For simplicity, we will refer to complete balanced $s$-ary trees simply as \emph{balanced $s$-ary trees} (except in \cref{sect:prebalanced}).
Observe that a balanced $s$-ary tree of height $h$ must have $n=\sum_{i=0}^h s^i$ vertices and diameter $2h$.

As typical in random graphs literature, we focus on asymptotic statements.
We state our asymptotic results in terms of the height $h$ of the trees, which then also yields asymptotic results with respect to the number of vertices. 
Our first result determines the sharp threshold for $\cG(n,r,d)$ to contain a spanning copy of the balanced $s$-ary tree for any fixed integer $s\ge 2$ and $n$ of the form $\sum_{i=0}^h s^i$.

\begin{theorem}\label{thm:treev1}
Fix positive integers\/ $s\geq2$ and\/ $d$.
Let\/ $h$ be a positive integer, and set\/ $n \coloneqq \sum_{i=0}^h s^i$.
Let\/ $T_h$ be the balanced\/ $s$-ary tree of height\/ $h$ (and on $n$ vertices).
Then,\/ $r^*\coloneqq\sqrt{d}/2h$ is the sharp threshold for\/ $\cG(n,r,d)$ to contain a copy of\/ $T_h$.
\end{theorem}

In fact, \Cref{thm:treev1} is a particular case of a similar result for a larger class of trees.
Given positive integers $h$ and $(s_i)_{i=1}^h$, we say that a tree $T$ is the \emph{balanced tree over the sequence $(s_i)_{i=1}^h$} if it has height $h$ and  for each $i\in\{1,\ldots,h\}$, every vertex of $T$ in $V_{i-1}$ has exactly $s_i$ children.
In particular, such a tree $T$ contains exactly $\sum_{i=0}^h \prod_{j=1}^i s_j$ vertices (where, by convention, the empty product equals~$1$). 
Moreover, if $M\geq2$ is an integer and $s_i\in \{2,\ldots,M\}$ for every $i\in \{1,\ldots,h\}$, we say that the balanced tree over the sequence $(s_i)_{i=1}^h$ is a balanced $M$-tree.

Our next result extends \cref{thm:treev1} to all balanced $M$-trees, as long as $M$ is not too large compared to $h$.

\begin{theorem}\label{thm:treev2}
Fix a positive integer\/ $d$, and let\/ $2\leq M = M(h) = o(h/\log h)$. 
Let\/ $(T_h)_{h\geq1}$ be a sequence of trees where\/ $T_h$ is a balanced\/ $M$-tree of height\/ $h$.
Then, \/$r^* \coloneqq \sqrt{d}/2h$ is the sharp threshold for\/ $\cG(|T_h|,r,d)$ to contain a copy of\/ $T_h$.
\end{theorem}

\Cref{thm:treev2} follows from a more general technical result (\cref{thm:tree}) which we state and prove in \cref{sec:proof}.

\subsection{Outline of proof}

In order to prove the upper bound for the threshold, we provide an efficient algorithm that finds an embedding of a balanced $M$-tree $T$ of height $h$ into $\cG(|T|,r,d)$ when $r \geq (1+\eps)r^*$ for any small (but fixed) $\eps > 0$.
This embedding is carried out in two steps. 
The first part of the proof consists of the analysis of an algorithm that embeds the first roughly $(1-\eps)h$ layers of $T$ into $\cG(|T|,r,d)$ (itself seen in its random embedding in $[0,1]^d$) in a fractal-like fashion.
The second part of the embedding relies on a version of Hall's theorem (for which efficient algorithmic implementations are well known).
The lower bound is comparably much easier and follows from the already mentioned comparison of the diameters of any connected graph and any of its spanning subgraphs.

\section{Auxiliary results}\label{sec:prelims}

We recall a variant of the famous Chernoff's bound (see, e.g.,~\cite[Corollary~2.3]{JLR}).

\begin{lemma}\label{lem:Chernoff}
Let\/ $X$ be a binomial random variable and\/ $\mu \coloneqq \mathbb{E}[X]$.
Then, for all\/ $0<\delta<1$, we have that 
\[\mathbb{P}[|X-\mu|\geq\delta\mu]\leq2\nume^{-\delta^2\mu/3}.\]
\end{lemma}

Next, we state a version of the celebrated Hall's theorem~\cite{Hall35} (see also \cite[Chapter~2.1]{Die17}) and deduce a corollary needed in the proof of Theorem~\ref{thm:tree}.

\begin{theorem}[\citet{Hall35}]\label{thm:hall}
A bipartite graph with parts\/ $A$ and\/ $B$ contains a matching of size\/ $|A|$ if and only if for every set\/ $S\subseteq A$, the number of neighbors of\/ $S$ in\/ $B$ is at least\/ $|S|$.
\end{theorem}

We refer to a copy of the bipartite graph $K_{1,k}$ in which $v$ constitutes the part containing exactly one vertex as a \emph{$k$-star with center $v$}. If we only specify that the center $v$ lies in some set $A$, we refer to this graph as a \emph{$k$-star with center in $A$}.

\begin{corollary}\label{cor:hall}
Fix a positive integer\/ $k$ and a bipartite graph\/ $G$ with parts\/ $A$ and\/ $B$ such that\/ $k|A| = |B|$.
Then, the vertices of\/ $G$ may be partitioned into\/ $|A|$ vertex-disjoint\/ $k$-stars with centers in\/ $A$ if and only if for every set\/ $S\subseteq A$, the number of neighbors of\/ $S$ in\/ $B$ is at least\/ $k|S|$.
\end{corollary}

\begin{proof}
The proof of the second statement from the first one is trivial.
Thus, we focus on the converse implication.
Let $A_1, \ldots, A_k$ be $k$ copies of the set $A$. For each $u\in A$ and each $i\in\{1,\ldots,k\}$, let $u_i$ be the copy of $u$ in $A_i$.
We define an auxiliary bipartite graph $\Gamma$ with parts $\overline{A} \coloneqq \bigcupdot_{i=1}^k A_i$ and $B$ where, for every $i\in \{1,\ldots,k\}$, $u\in A$ and $v\in B$, we have $u_iv\in E(\Gamma)$ if and only if $uv\in E(G)$.

Now, for any set $S\subseteq \overline{A}$, let $\alpha(S)\subseteq A$ be the set of all vertices $u\in A$ such that, for some $j\in\{1,\ldots,k\}$, we have $u_j\in S$. Then, for every $S\subseteq \overline{A}$, the neighborhoods of $S$ in $\Gamma$ and $\alpha(S)$ in $G$ coincide.
This means that the neighborhood of $S$ in $\Gamma$ contains at least $k|\alpha(S)|\ge |S|$ vertices.
Hence, an application of Hall's theorem ensures the existence of a perfect matching in $\Gamma$. Finally, identifying the sets $A_1, \ldots, A_k$ shows that one may partition the original graph into $|A|$ $k$-stars with centers in $A$, as desired.
\end{proof}

\begin{remark}\label{rem:polynomial}
If the condition of Hall's theorem is satisfied, there exist efficient algorithms for finding a matching in a bipartite graph.
A classic polynomial time algorithm for iterative construction of such a matching is based on augmenting paths; see, e.g.,~\cite[Section~2.1]{Die17}. 
Hopcroft and Karp's algorithm~\cite{HK73} improved on the classic algorithm by processing several augmenting paths at a time.
Later, Chwa and Kim~\cite{CK87}, and Goldberg, Plotkin and Vaidya~\cite{GPV93} came up with even faster parallel algorithms.
As a consequence, under the assumptions of \cref{cor:hall}, finding a family of $|A|$ disjoint stars with centers in $A$ can also be done in polynomial time.
\end{remark}

\section{Balanced trees in random geometric graphs}\label{sec:proof}

In this section, we state and prove the more general theorem from which \cref{thm:treev1,thm:treev2} follow.
This result is stated in terms of the \emph{random geometric graph sequence} defined as follows.
Let $r$ be a positive real number and let $d$ be a positive integer.
Given a set of points $\mathcal{X}\subseteq[0,1]^d$, the geometric graph $G(\mathcal{X}, r, d)$ is the graph with vertex set $\mathcal{X}$ and edge set $\{xy:x,y\in\mathcal{X}, x\neq y,\lVert x-y\rVert\leq r\}$, where $\lVert\cdot\rVert$ denotes the Euclidean norm.
Let $(X_i)_{i\geq1}$ be a sequence of independent uniform random variables on $[0,1]^d$.
The \emph{random geometric graph sequence} of radius $r$ is a sequence of graphs $\widehat{\cG}(r,d)=(G_n)_{n\ge 1}$ where $G_n=G(\{X_1,\ldots,X_n\},r,d)$.
In particular, note that $G_n$ is distributed as $\cG(n,r,d)$, but the different graphs in the sequence are not independent of each other.

Given a probability space $(\Omega, \cF, \mathbb{P})$ and events $E_h\in \cF$ for all $h\ge 1$, we say that $(E_h)_{h\ge 1}$ holds \emph{asymptotically almost surely}, or \emph{a.a.s.}\ for short, if $\mathbb{P}(E_h)\to 1$ as $h\to \infty$.
For real numbers $a,b,c$ with $c>0$, we write $a=b\pm c$ to mean that $a\in[b-c,b+c]$.

Our main result can now be stated as follows:

\begin{theorem}\label{thm:tree}
Fix a positive integer\/ $d$ and a real number\/ $\epsilon\in(0,1)$.
For each positive integer\/ $h$, let\/ $2\leq M = M(h) = o(h/{\log h})$, set\/ $r^*=r^*(h)\coloneqq \sqrt{d}/{2h}$, and define two events\/ $\cE_h$ and\/ $\cF_h$ as follows.
Consider the random geometric graph sequence\/ $\widehat{\cG}(r,d)=(G_n)_{n\geq1}$.
Let\/ $\cE_h$ denote the event that each balanced\/ $M$-tree\/ $T$ of height\/ $h$ appears as a subgraph of\/ $G_{|T|}$.
Let\/ $\cF_h$ denote the event that, for each balanced\/ $M$-tree\/ $T$ of height\/ $h$, the graph\/ $G_{|T|}$ does not contain\/ $T$ as a subgraph.
Then, the following two statements hold.
\begin{enumerate}[label=$(\mathrm{\roman*})$]
    \item\label{thm:item1} If\/ $r\leq(1-\epsilon)r^*$, then\/ $\cF_h$ holds a.a.s.
    \item\label{thm:item2} If\/ $r\geq(1+\epsilon)r^*$, then\/ $\cE_h$ holds a.a.s.
\end{enumerate}
\end{theorem}

Note that Theorem~\ref{thm:tree} is stronger than determining the sharp threshold of one single balanced $T$-tree: indeed, our events $\cE_h$ and $\cF_h$ consider all balanced $M$-trees of height $h$ simultaneously.

In the remainder of this section we prove \cref{thm:tree}.
The proof of \ref{thm:item1} is fairly simple, and most of the work will be devoted to proving \ref{thm:item2}.
For this, we are going to provide an algorithm which, under a mild assumption on the distribution of the points finds a copy of any given balanced $M$-tree $T$ of height $h$ in $G_{|T|}$ in polynomial time.
We will later show that a.a.s.\ all sufficiently large graphs appearing in the random geometric graph sequence satisfy this point-distribution property.
Before describing the embedding algorithm (see \cref{thm:deterministic} below), we introduce some notation and definitions.

Fix $d\in \mathbb N$ and $\eps\in (0,1)$.
Given $h\in\mathbb{N}$, let $r^*$ be defined as in \cref{thm:tree}.
Here and below, we refer to balanced $M$-trees of height $h$ simply as $M$-trees.
All cubes $q\subseteq \mathbb R^d$ considered throughout will be closed and axis-parallel (that is, they can all be obtained by some homothety from $[0,1]^d$).
For a set $A\subseteq \mathbb R^d$, we denote by $\partial A$ the topological boundary of $A$.
Given a set $\mathcal{W}$ of subsets of $\mathbb{R}^d$, we denote $\partial\mathcal{W}\coloneqq\bigcup_{A\in\mathcal{W}}\partial A$.
For any $x\in\mathbb{R}^d$ and $r>0$, we write $B_x(r)$ to denote the closed ball of radius $r$ and center $x$, that is, the set of all $y\in\mathbb{R}^d$ such that $\lVert x-y\rVert\le r$.

The embedding algorithm will be based on a certain tessellation of $[0,1]^d$.
Let $r \coloneqq (1+\eps) r^*$.
Given integers $M\geq2$ and $s\in \{2,\ldots,M\}$, let $k_s$ be the smallest integer which simultaneously satisfies that
\begin{equation}\label{equa:treebound1}
    \sqrt{d} s^{1-k_s}\le r
\end{equation}
and
\begin{equation}\label{equa:treebound2}
    \sqrt{d} s^{-k_s} < \frac{\eps r^*}{8}.
\end{equation}
Note that $k_s = \Theta({\log h}/{\log s}) = O(\log h)$ and that the sequence $(k_s)_{s=2}^M$ is non-increasing.
Then, let $\cW_s$ be the tessellation of $[0,1]^d$ into $s^{k_s d}$ congruent closed axis-parallel cubes.
We remark that \eqref{equa:treebound1} ensures that, for every point $x\in \mathbb R^d$, any cube of side length $s^{1-k_s}$ that contains $x$ is itself contained in $B_x(r)$.
Moreover,~\eqref{equa:treebound2} is imposed to ensure that we can nicely approximate a ball of radius $r^*$ using cubes in $\cW_s$: more precisely, for the ball of radius $(1+{3\eps}/{4})r^*$ centered at some point $x$, the cubes in $\cW_s$ intersecting $\partial B_x((1+{3\eps}/{4})r^*)$ are entirely contained in $B_x(r)\setminus B_x((1+{\eps}/{2})r^*)$; see Figure~\ref{fig 2}.
Suppose now that $M$ and $h$ are such that $m' \coloneqq d k_2 M\leq h$.
Consider a sequence $(s_i)_{i=1}^h$ with $s_i\in\{2,\ldots,M\}$ for every $i\in \{1,\ldots,h\}$, and the balanced tree $T$ over this sequence.
By the pigeonhole principle, there is an integer in $\{2,\ldots,M\}$ appearing at least $d k_2$ times among $(s_i)_{i=1}^{m'}$. 
Let $s(T)$ denote one such integer.\COMMENT{there may be multiple choices for $s(T)$; in such a case, we simply choose $s(T)$ arbitrarily among the possible values.}

Our tree embedding algorithm is encoded into (the proof of) the following theorem:

\begin{theorem}\label{thm:deterministic}
For each positive integer\/ $h$, let\/ $2\leq M=M(h)=o(h/\log h)$.
For every positive integer\/ $d$ and\/ $\epsilon\in(0,1)$, there exists some\/ $h_0\in\mathbb{N}$ such that the following holds for all\/ $h\geq h_0$:

Let\/ $r^*\coloneqq\sqrt{d}/2h$.
Let\/ $(s_i)_{i=1}^h$ be a sequence where\/ $s_i\in\{2,\ldots,M\}$, let\/ $T$ be the balanced tree over the sequence\/ $(s_i)_{i=1}^h$, and let\/ $s\coloneqq s(T)$.
Let\/ $\mathcal{X}\subseteq[0,1]^d\setminus\partial\mathcal{W}_s$ be a set of\/ $|T|$ points such that each\/ $q\in\mathcal{W}_{s}$ contains\/ $s^{-k_s d}|T|\pm|T|^{2/3}$ points.
Then,\/ $G(\cX,(1+\epsilon)r^*, d)$ contains a copy of\/ $T$.
\end{theorem}

\begin{remark}
    The boundary condition $\mathcal{X}\subseteq[0,1]^d\setminus\partial\mathcal{W}_s$ is a minor technical restriction clearly satisfied by many point processes such as the Poisson point process and only needed as an artifact of the proof.
\end{remark}

\begin{proof}[Proof of \cref{thm:deterministic}]
First, we note that, by adjusting the value of $h_0$, the parameter $s$ in the statement is well defined.
Indeed, since $k_2=O(\log h)$ and $M=o(h/\log h)$, it follows that $m'=o(h)$ and, in particular, $s_{m'}$ is defined for all sufficiently large $h$.
Specifically, we assume that $h_0$ is sufficiently large so that for all $h\geq h_0$ we have
\begin{equation}\label{equa:treebound4}
    m'\leq \epsilon h/20.
\end{equation}
Moreover, letting $k\coloneqq k_s$, we also pick $h_0$ sufficiently large so that the following hold for all $h\geq h_0$:
\begin{align}
    k&\leq \epsilon h/12,\label{equa:treebound7}\\
    2^{-\epsilon h/5}&\leq s^{-kd}/2,\label{equa:treebound5}\\
    |T|^{2/3}&<s^{-2kd}|T|/2.\label{equa:treebound6}
\end{align}

Let $r \coloneqq (1+\eps) r^*$, and consider the tessellation $\mathcal{W}_s$ described before the statement of \cref{thm:deterministic}.
For each $\ell\in \{0,\ldots,k\}$, let $\cS_{\ell}$ be the tessellation of $[0,1]^d$ into $s^{\ell d}$ congruent closed axis-parallel cubes obtained by combining the cubes of $\cW_s$ into groups of size $s^{(k-\ell) d}$.
In particular, $\cS_0 = \{[0,1]^d\}$ and $\cS_k = \cW_s$.
For a cube $q\subseteq[0,1]^d$, we denote its center by $c(q)$.
Moreover, for every $i,j\in \{0,\ldots,k\}$ with $i < j$ and any cube $q\in \cS_i$, let $\sigma_j(q)$ denote the set of $s^d$ subcubes of $q$ in $\cS_j$ that form the axis-parallel cube of side length $s^{1-j}$ and center $c(q)$.
In particular, $\sigma_{i+1}(q)$ is a tessellation of $q$ into $s^d$ subcubes in $\cS_{i+1}$, and for each $j\in\{i+2,\ldots,k\}$, the set $\sigma_j(q)$ is obtained from $\sigma_{j-1}(q)$ by homothety with center $c(q)$ and ratio $1/s$.
For example, Figure~\ref{fig setup}~(a) illustrates $\cS_1$, $\cS_2$ and $\cS_3$ when $s = 3$ and $d=2$, and Figure~\ref{fig setup}~(b) depicts the set of cubes $\sigma_2(q)$ for $q = [0,1]^2$.
Note that we sometimes abuse notation and identify $\sigma_j(q)$ with its geometric realization; in particular, we identify $\bigcup_{p\in \sigma_j(q)} p$ with $\sigma_j(q)$ itself.

\begin{figure}
\centering
\resizebox{0.6\textwidth}{!}{
\begin{tikzpicture}[scale=1,line cap=round,line join=round,x=1cm,y=1cm]

\foreach \i in {0,9,18,27}{
\draw [line width=15pt] (0,\i)-- (27,\i);
\draw [line width=15pt] (\i,0)-- (\i,27);
}
\foreach \i in {3,6,12,15,21,24}{
\draw [line width=7pt,red] (0,\i)-- (27,\i);
\draw [line width=7pt,red] (\i,0)-- (\i,27);
}
\foreach \i in {1,2,4,5,7,8,10,11,13,14,16,17,19,20,22,23,25,26}{
\draw [line width=2pt] (0,\i)-- (27,\i);
\draw [line width=2pt] (\i,0)-- (\i,27);
}
\draw (13.5,-3) node[font={\fontsize{60pt}{12}\selectfont}] {(a)};
\end{tikzpicture}\hspace{11cm}% 
\begin{tikzpicture}[line join=round]
\draw [line width=15pt] (0,0) -- (0,27) -- (27,27) -- (27,0) -- (0,0) -- (0,27);
\draw[anchor=north west] (17,3.5) node[font={\fontsize{60pt}{12}\selectfont}] {$q=[0,1]^2$};
\draw (13.5,-3) node[font={\fontsize{60pt}{12}\selectfont}] {(b)};
\draw[anchor=north west] (8.7,8) node[font={\fontsize{60pt}{12}\selectfont}] {$\sigma_2(q)\subseteq\mathcal{S}_2$};
\foreach \i in {9,12,15,18}{
\draw [line width=7pt,red] (9,\i) -- (18,\i);
\draw [line width=7pt,red] (\i,9) -- (\i,18);
}
\draw [line width=7pt,red] (9,9) -- (18,9) -- (18,18) -- (9,18) -- (9,9) -- (18,9);
\end{tikzpicture}
}
\caption{Figure (a) shows a tessellation of $[0,1]^2$ into $3^{2\cdot 3}$ subcubes.
Figure (b) shows an example of $\sigma_2$ applied to $q=[0,1]^2$.}
\label{fig setup}
\end{figure}
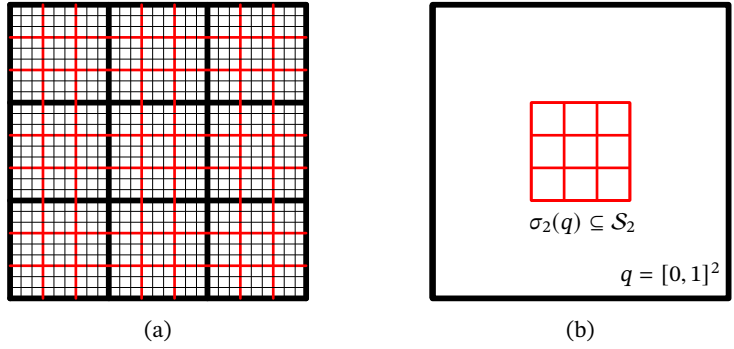

Let $G \coloneqq G(\mathcal{X},r, d)$.
Let us give a general description of the algorithm that we propose for finding a copy of $T$ in $G$.
Let $v_0$ be the root of $T$, and consider the partition of the vertices of $T$ into layers $V_0,V_1,\ldots,V_h$.
We embed the layers of $T$ into $G$ one at a time, starting from the root.
Our algorithm has three main parts which we call \emph{subroutines}.
The first subroutine lasts for $m'+1$ steps.
It embeds the layers $V_0, \ldots, V_{m'-1}$ of $T$ into an arbitrary cube $q_0\in \sigma_k([0,1]^d)$, and then all vertices in $V_{m'}$ are evenly distributed among the $s^d$ cubes in $\sigma_k([0,1]^d)$. 
The second subroutine is used to distribute the vertices in the next roughly $(1-\eps)h$ layers as evenly as possible in $[0,1]^d$; more precisely, for some $m\leq(1-\eps/4)h$ (see \cref{lem:nb of steps} below), we manage to ensure that every cube $q\in\cS_k$ contains exactly the same number of vertices of $V_{m'+m}$.
Once this point is reached, in the third subroutine, the remaining layers are embedded into $G$ by using Corollary~\ref{cor:hall}.
In particular, we associate a vertex $v\in V_{m'+m}$ to each of the vertices of $G$ that are still not covered by $T$ in such a way that $v$ and the vertices associated with $v$ form a clique whose size is the same for every $v\in V_{m'+m}$.
Note that this is sufficient to guarantee that the embedding can be extended.

Throughout, we refer to the embedding of one of the layers of $T$ as a \emph{step} of the algorithm (for simplicity, we assume that $V_0$ is embedded at the $0$-th step).
For each $i\in\{0,\ldots,m'+m\}$, at the end of the $i$-th step, we call a vertex of $G$ into which a vertex on layer $V_i$ has been embedded \emph{active}.
Moreover, we refer to the vertices of $G$ into which no vertex of $T$ has been embedded as \emph{unseen}.
For simplicity of notation, once a vertex of $T$ has been embedded into $G$, we often interchangeably use the same notation to refer to either of the two vertices.
We note here that, since $m \leq (1-\eps/4)h$, it follows that the total number of vertices embedded during the first two subroutines is\COMMENT{Intuition for the first inequality: we have to subtract the sizes of the last few layers, and these are always very large. In particular, each layer is at least twice as big as the previous. This means that each layer contains at least as many vertices as all the previous layers together. Thus, by counting from the end, when removing the last layer we reduce the size to at most a half, then at most a quarter, and so on.}
\begin{equation}\label{equa:nbvertices}
    \sum_{i=0}^{m'+m} \prod_{j=1}^i s_j \le 2^{-(h-m-m')} |T|\leq 2^{-\epsilon h/5}|T|\leq\frac{|T|}{2s^{kd}}.
\end{equation}
(Here, the first inequality follows from the fact that $s_i\ge 2$ for all $i\in\{1,\ldots,h\}$, the second inequality follows by combining the upper bound on $m$ given by \cref{lem:nb of steps} below with \eqref{equa:treebound4}, and the third follows by \eqref{equa:treebound5}.)
Therefore, by \eqref{equa:treebound6} and the condition on the point distribution in the statement, there are sufficiently many vertices in each $q\in\mathcal{S}_k$ at our disposal throughout the first two subroutines, which guarantees that the choices we make below can indeed be carried out. 

The procedure that we follow to embed the layers of $T$ into the cubes of $\cS_k$ throughout the first subroutine is very simple.
We fix an arbitrary cube $q_0\in\sigma_k([0,1]^d)$.
Then, for each $i\in\{0,\ldots,m'-1\}$, we arbitrarily embed all vertices of $V_i$ into $q_0$.
Finally, the vertices of $V_{m'}$ are split (arbitrarily) into $s^d$ sets of the same size, and each of these sets is embedded entirely into a different cube $q\in\sigma_k([0,1]^d)$.
Note that this is possible since $s^d$ divides $\prod_{i=1}^{m'} s_i$ by the definition of $s$, and since the diameter of $\sigma_k([0,1]^d)$ is at most $r$ by~\eqref{equa:treebound1}.

Let us now describe the algorithm that we use to embed the layers of $T$ into the cubes of $\cS_k$ throughout the second subroutine.
The $m$ steps of the second subroutine are grouped into $k-1$ different blocks.
For each $\ell\in \{1,\ldots,k-1\}$, we proceed iteratively as follows:
Suppose that, at the beginning of the $\ell$-th block, the algorithm has reached a configuration in which every cube $q\in \cS_{\ell-1}$ contains the same number of active vertices, and that these are equally distributed among all subcubes in $\sigma_k(q)$ (note that this is verified in the case $\ell=1$).
Then, for each $q\in \cS_{\ell-1}$, we proceed to distributing the descendants of the currently active vertices in a way that we embed them at non-decreasing distances from $\sigma_k(q)$ as described in the sequel.

\vspace{1em}
\textbf{Iteration:} 
Define $\phi_{\ell}\colon\sigma_k(q)\to \sigma_{\ell}(q)$ as the bijection obtained by homothety with center $c(q)$ and ratio $s^{k-\ell}$.
Note that $\phi_{\ell}$ depends on the cube $q$, which is fixed in our argument.
To each cube $p\in \sigma_k(q)$ we associate a sequence of (not necessarily distinct) cubes $(p_1, \ldots, p_t)$ in $\cS_k$, for some appropriately chosen $t$ which does not depend on $p$, which satisfies that
\begin{enumerate}[label=$(\mathrm{P}\arabic*)$]
    \item\label{prop1} $p_1 = p$;
    \item\label{prop2} $p_t\in \sigma_k(\phi_{\ell}(p))$, and
    \item\label{prop3} for all $i\in\{1,\ldots,t-1\}$ we have $\lVert c(p_{i+1}) - c(p_i)\rVert\le \left(1+7\eps/8\right) r^*$.
\end{enumerate}
Note that \ref{prop3} together with the triangle inequality and \eqref{equa:treebound2} ensures that, for every $i\in \{1,\ldots,t-1\}$ and every choice of points $x\in p_i$ and $y\in p_{i+1}$, we have
\begin{equation}\label{equa:treebound3}
    \lVert x-y\rVert\leq\lVert x-c(p_i)\rVert+\lVert c(p_i)-c(p_{i+1})\rVert+\lVert c(p_{i+1}) - y\rVert\leq\frac{\eps}{16} r^* + \left(1+\frac{7\eps}{8}\right) r^* + \frac{\eps}{16} r^* = r.
\end{equation}
Hence, if at some point there is an active vertex $v$ in $p_i$, it is possible to embed all its children in $p_{i+1}$.

Let us prove that the sequences of cubes described above can indeed be constructed.
We will show the construction when starting from a cube $p\in \sigma_k(q)$ whose center is at a maximum distance from $c(q)$ among all cubes in $\sigma_k(q)$ (note that, in this case, $p$ must contain a corner of the larger cube $\sigma_k(q)$).
Observe that this choice of $p$ also maximizes the distance between $c(p)$ and $c(\phi_{\ell}(p))$.
For every other $p'\in \sigma_k(q)$, the construction is the same except that consecutive cubes are chosen at a smaller distance from each other (or even coincide in some cases).

Fix $p \in\sigma_k(q)$ as above.
We construct the sequence of cubes $(p_1, \ldots, p_t)$ in $\cS_k$ satisfying the desired properties \ref{prop1}--\ref{prop3}.
We begin by setting $p_1\coloneqq p$.
For each $i\geq1$, while $\lVert c(p_i) - c(\phi_{\ell}(p))\rVert > (1+7\eps/8)r^*$, we choose a cube $p_{i+1} \in\mathcal{S}_k$ such that
\begin{equation}\label{eq:step1}
    \lVert c(p_{i+1}) - c(\phi_{\ell}(p))\rVert\leq\lVert c(p_i) - c(\phi_{\ell}(p))\rVert - \left(1+\frac{5\eps}{8}\right) r^*
\end{equation}
and
\begin{equation}\label{eq:step2}
    \lVert c(p_{i+1}) - c(p_i)\rVert\le \left(1+\frac{7\eps}{8}\right) r^*.
\end{equation}
Finally, when we reach a point where $\lVert c(p_i) - c(\phi_{\ell}(p))\rVert \le (1+7\eps/8)r^*$, we set $t\coloneqq i+1$ and choose $p_t$ as a cube in $\sigma_k(\phi_{\ell}(p))$ at smallest distance to $c(p_i)$, ties being broken arbitrarily.\COMMENT{Note actually there can be no ties, so the definition is unique.}
Note that~\eqref{equa:treebound3} holds for $i=t-1$ and this choice of $p_t$.
It remains to prove that a choice as prescribed by \eqref{eq:step1} and \eqref{eq:step2} can indeed be carried out.

\begin{figure}
\centering
\resizebox{0.9\textwidth}{!}{
\begin{tikzpicture}[scale=0.7,line cap=round,line join=round,x=1cm,y=1cm]
\clip(-0.5,-3.727681011037504) rectangle (23.60020225817874,9.753493408727518);
\draw [line width=0.8pt] (0,-1) circle (6cm);
\draw [line width=0.8pt] (0,-1) circle (8cm);
\draw [line width=0.8pt] (0,-1) circle (10cm);
\draw [line width=0.8pt] (0,-1)-- (22,8);
\draw [line width=0.8pt] (0,-1) circle (9cm);
\draw [line width=0.8pt] (0,-1)-- (6,-1);
\draw [line width=0.8pt] (8,2.6)-- (8,2);
\draw [line width=0.8pt] (8,2)-- (8.6,2);
\draw [line width=0.8pt] (8.6,2)-- (8.6,2.6);
\draw [line width=0.8pt] (8.6,2.6)-- (8,2.6);
\begin{scriptsize}
\draw [fill=black] (0,-1) circle (1pt);
\draw[color=black] (-0.02789889731862069,-1.3) node {$c(p_i)$};
\draw[color=black] (8.2,2.17) node {$w$};
\draw[color=black] (22.2,7.6) node {$c(\phi_{\ell}(p))$};
\draw [fill=black] (22,8) circle (1pt);
\draw[color=black] (3.054027340354948,-1.3) node {$r^*$};
\draw [fill=black] (8.329922605851092,2.407695611484537) circle (1pt);
\end{scriptsize}
\end{tikzpicture}
}
\caption{A one-step transition from $p_i$ to $p_{i+1}$ in two dimensions.
The four circles in the figure are all centered at $c(p_i)$ and have radii $r^*$, $(1+{\eps}/{2})r^*$, $(1+{3\eps}/{4})r^*$, and $r$.
Note that $p_{i+1}$ is the square containing $w$; its center as well as the boundary of the square $p_i$ are not presented for reasons of clarity of the figure.}
\label{fig 2}
\end{figure}
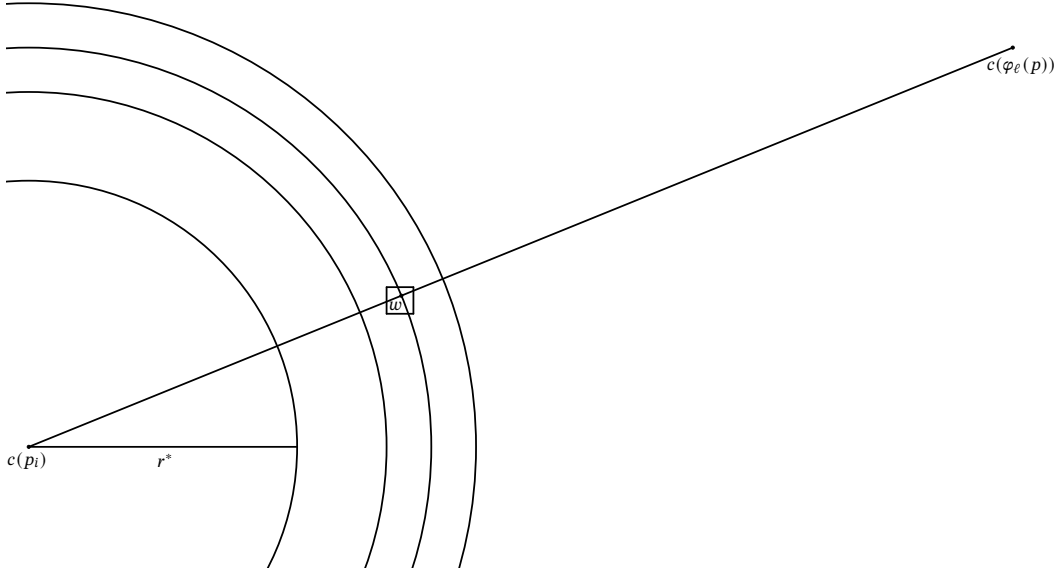

Assume we have already defined $p_i$ so that it satisfies $\lVert c(p_i) - c(\phi_{\ell}(p))\rVert > (1+7\eps/8)r^*$.
Then, choose $p_{i+1} \in\mathcal{S}_k$ to be a cube containing the point $w$ on the segment $c(p_i)c(\phi_{\ell}(p))$ at distance exactly $(1+{3\eps}/{4})r^*$ from $c(p_i)$ (if more than one such cube exists, we choose one arbitrarily); see Figure~\ref{fig 2}.
Then, the triangle inequality and \eqref{equa:treebound2} imply that
\begin{align*}
\lVert c(p_{i+1}) - c(\phi_{\ell}(p))\rVert
&\leq\lVert w-c(\phi_{\ell}(p))\rVert + \lVert c(p_{i+1})-w\rVert\\
&=\lVert c(p_i)-c(\phi_{\ell}(p))\rVert - \lVert c(p_i)-w\rVert + \lVert c(p_{i+1})-w\rVert\\
&\leq\lVert c(p_i)-c(\phi_{\ell}(p))\rVert - \left(1+\frac{3\eps}{4}\right)r^* + \frac{\eps}{16} r^* \\
&\leq\lVert c(p_i) - c(\phi_{\ell}(p))\rVert - \left(1+\frac{5\eps}{8}\right) r^*
\end{align*}
and
\begin{align*}
\lVert c(p_i) - c(p_{i+1})\rVert\leq\lVert c(p_i) - w\rVert+\lVert w-c(p_{i+1})\rVert\leq\left(1+\frac{3\eps}{4}\right)r^* + \frac{\eps}{16} r^* \le \left(1+\frac{7\eps}{8}\right) r^*,
\end{align*}
so \eqref{eq:step1} and \eqref{eq:step2} are verified.

\begin{figure}
\centering
\resizebox{0.65\textwidth}{!}{
\begin{tikzpicture}[scale=1.55,line cap=round,line join=round,x=1cm,y=1cm]
%\clip(-11.3,-2) rectangle (5.5,6.1);
\draw [line width=0.4pt] (-10,6)-- (-10,-2);

\draw [line width=2.5pt] (-6,6)-- (-6,2);
\draw [line width=2.5pt] (-6,2)-- (-2,2);
\draw [line width=2.5pt] (-2,2)-- (-2,6);
\draw [line width=2.5pt] (-2,6)-- (-6,6);

\draw [line width=2.5pt] (-6,2.5)--(-5.5,2.5);
\draw [line width=2.5pt] (-5.5,2)--(-5.5,2.5);

\draw [line width=0.4pt] (-10,-2)-- (-2,-2);
\draw [line width=0.4pt] (-2,-2)-- (-2,6);
\draw [line width=0.4pt] (-2,6)-- (-10,6);
\draw [line width=0.4pt] (-10,2)-- (-2,2);
\draw [line width=0.4pt] (-6,6)-- (-6,-2);
\draw [line width=0.4pt] (-4,6)-- (-4,-2);
\draw [line width=0.4pt] (-8,6)-- (-8,-2);
\draw [line width=0.4pt] (-10,0)-- (-2,0);
\draw [line width=0.4pt] (-10,-1)-- (-2,-1);
\draw [line width=0.4pt] (-10,4)-- (-2,4);
\draw [line width=0.4pt] (-10,3)-- (-2,3);
\draw [line width=0.4pt] (-10,5)-- (-2,5);
\draw [line width=0.4pt] (-9,6)-- (-9,-2);
\draw [line width=0.4pt] (-10,1)-- (-2,1);
\draw [line width=0.4pt] (-7,6)-- (-7,-2);
\draw [line width=0.4pt] (-5,6)-- (-5,-2);
\draw [line width=0.4pt] (-3,6)-- (-3,-2);
\draw [line width=0.4pt] (-9.5,6)-- (-9.5,-2);
\draw [line width=0.4pt] (-8.5,6)-- (-8.5,-2);
\draw [line width=0.4pt] (-7.5,6)-- (-7.5,-2);
\draw [line width=0.4pt] (-6.5,6)-- (-6.5,-2);
\draw [line width=0.4pt] (-5.5,6)-- (-5.5,-2);
\draw [line width=0.4pt] (-4.5,6)-- (-4.5,-2);
\draw [line width=0.4pt] (-3.5,6)-- (-3.5,-2);
\draw [line width=0.4pt] (-2.5,6)-- (-2.5,-2);
\draw [line width=0.4pt] (-10,-1.5)-- (-2,-1.5);
\draw [line width=0.4pt] (-10,-0.5)-- (-2,-0.5);
\draw [line width=0.4pt] (-10,0.5)-- (-2,0.5);
\draw [line width=0.4pt] (-10,1.5)-- (-2,1.5);
\draw [line width=0.4pt] (-10,2.5)-- (-2,2.5);
\draw [line width=0.4pt] (-10,3.5)-- (-2,3.5);
\draw [line width=0.4pt] (-10,4.5)-- (-2,4.5);
\draw [line width=0.4pt] (-10,5.5)-- (-2,5.5);

\draw [color=red, line width=2.5pt] (-6-4,6-4)-- (-6-4,2-4);
\draw [color=red, line width=2.5pt] (-6-4,2-4)-- (-2-4,2-4);
\draw [color=red, line width=2.5pt] (-2-4,2-4)-- (-2-4,6-4);
\draw [color=red, line width=2.5pt] (-2-4,6-4)-- (-6-4,6-4);

\draw [color=red, line width=2.5pt] (-6,1.5)--(-6.5,1.5);
\draw [color=red, line width=2.5pt] (-6.5,1.5)--(-6.5,2);

\draw (-6.4+0.5,2.5-0.1) node[anchor=north west] {0};
\draw (-6.4+0.5,2-0.1) node[anchor=north west] {0};
\draw (-6.4,2-0.1) node[anchor=north west] {0};
\draw (-6.4,2.5-0.1) node[anchor=north west] {0};
\draw (-5.54+0.05,3-0.08) node[anchor=north west] {1,6};
\draw (-7.04+0.05,3-0.08) node[anchor=north west] {1,6};
\draw (-7.04+0.05,1.5-0.08) node[anchor=north west] {1,6};
\draw (-5.54+0.05,1.5-0.08) node[anchor=north west] {1,6};
\draw (-5.14+0.11,3.5-0.11) node [anchor=north west] {\tiny{$2,5,6$}};
\draw (-7.64+0.11,3.5-0.11) node[anchor=north west] {\tiny{$2,5,6$}};
\draw (-7.64+0.11,1-0.11) node[anchor=north west] {\tiny{$2,5,6$}};
\draw (-5.14+0.11,1-0.11) node[anchor=north west] {\tiny{$2,5,6$}};
\draw (-4.54+0.03,4-0.07) node[anchor=north west] {3,4};
\draw (-8.04+0.03,4-0.07) node[anchor=north west] {3,4};
\draw (-8.04+0.03,0.5-0.07) node[anchor=north west] {3,4};
\draw (-4.54+0.03,0.5-0.07) node[anchor=north west] {3,4};
\draw (-4+0.1,4.5-0.085) node[anchor=north west] {4};
\draw (-4.5+0.1,4.5-0.085) node[anchor=north west] {4};
\draw (-4+0.1,4-0.085) node[anchor=north west] {4};
\draw (-8.5+0.1,4.5-0.085) node[anchor=north west] {4};
\draw (-8+0.1,4.5-0.085) node[anchor=north west] {4};
\draw (-8.5+0.1,4-0.085) node[anchor=north west] {4};
\draw (-4+0.1,0.5-0.085) node[anchor=north west] {4};
\draw (-4.5+0.1,0-0.085) node[anchor=north west] {4};
\draw (-4+0.1,0-0.085) node[anchor=north west] {4};
\draw (-8.5+0.1,0-0.085) node[anchor=north west] {4};
\draw (-8+0.1,0-0.085) node[anchor=north west] {4};
\draw (-8.5+0.1,0.5-0.085) node[anchor=north west] {4};
\draw (-3.5,-0.5-0.08) node[anchor=north west] {5,6};
\draw (-5,-0.5-0.08) node[anchor=north west] {5,6};
\draw (-3.5,1-0.08) node[anchor=north west] {5,6};
\draw (-7.5,-0.5-0.08) node[anchor=north west] {5,6};
\draw (-9,-0.5-0.08) node[anchor=north west] {5,6};
\draw (-9,1-0.08) node[anchor=north west] {5,6};
\draw (-9,3.5-0.08) node[anchor=north west] {5,6};
\draw (-9,5-0.08) node[anchor=north west] {5,6};
\draw (-7.5,5-0.08) node[anchor=north west] {5,6};
\draw (-5,5-0.08) node[anchor=north west] {5,6};
\draw (-3.5,5-0.08) node[anchor=north west] {5,6};
\draw (-3.5,3.5-0.08) node[anchor=north west] {5,6};
\draw (-3.5+0.1,5.5-0.08) node[anchor=north west] {6};
%\draw (-3,5.5) node[anchor=north west] {9,10};
\draw (-3+0.1,5-0.08) node[anchor=north west] {6};
\draw (-3+0.12,5.5-0.08) node[anchor=north west] {6};
\draw (-5.5+0.12,5.5-0.08) node[anchor=north west] {6};
\draw (-5.5+0.1,5-0.08) node[anchor=north west] {6};
\draw (-3+0.1,3.5-0.08) node[anchor=north west] {6};
\draw (-3.5+0.1,3-0.08) node[anchor=north west] {6};
\draw (-3+0.12,3-0.08) node[anchor=north west] {6};
\draw (-7.5+0.1,5.5-0.08) node[anchor=north west] {6};
\draw (-7+0.12,5.5-0.08) node[anchor=north west] {6};
\draw (-7+0.1,5-0.08) node[anchor=north west] {6};
\draw (-9+0.1,5.5-0.08) node[anchor=north west] {6};
\draw (-9.5+0.12,5.5-0.08) node[anchor=north west] {6};
\draw (-9.5+0.1,5-0.08) node[anchor=north west] {6};
\draw (-9+0.1,1.5-0.08) node[anchor=north west] {6};
\draw (-9.5+0.12,1.5-0.08) node[anchor=north west] {6};
\draw (-9.5+0.1,1-0.08) node[anchor=north west] {6};
\draw (-9.5+0.1,-0.5-0.08) node[anchor=north west] {6};
\draw (-9.5+0.12,-1-0.08) node[anchor=north west] {6};
\draw (-9+0.1,-1-0.08) node[anchor=north west] {6};
\draw (-9+0.1,3-0.08) node[anchor=north west] {6};
\draw (-9.5+0.12,3-0.08) node[anchor=north west] {6};
\draw (-9.5+0.1,3.5-0.08) node[anchor=north west] {6};
\draw (-7.5+0.1,-1-0.08) node[anchor=north west] {6};
\draw (-7+0.1,-0.5-0.08) node[anchor=north west] {6};
\draw (-7+0.12,-1-0.08) node[anchor=north west] {6};
\draw (-5.5+0.1,-0.5-0.08) node[anchor=north west] {6};
\draw (-5.5+0.12,-1-0.08) node[anchor=north west] {6};
\draw (-5+0.1,-1-0.08) node[anchor=north west] {6};
\draw (-3.5+0.1,-1-0.08) node[anchor=north west] {6};
\draw (-3+0.12,-1-0.08) node[anchor=north west] {6};
\draw (-3+0.1,-0.5-0.08) node[anchor=north west] {6};
\draw (-3+0.1,1-0.08) node[anchor=north west] {6};
\draw (-3+0.12,1.5-0.08) node[anchor=north west] {6};
\draw (-3.5+0.1,1.5-0.08) node[anchor=north west] {6};
\draw (-5+0.1,1.5-0.08) node[anchor=north west] {6};
\draw (-5.5+0.1,1-0.08) node[anchor=north west] {6};
\draw (-7+0.1,1-0.08) node[anchor=north west] {6};
\draw (-7.5+0.1,1.5-0.08) node[anchor=north west] {6};
\draw (-7.5+0.1,3-0.08) node[anchor=north west] {6};
\draw (-7+0.1,3.5-0.08) node[anchor=north west] {6};
\draw (-5.5+0.1,3.5-0.08) node[anchor=north west] {6};
\draw (-5+0.1,3-0.08) node[anchor=north west] {6};
\draw (-5+0.1,5.5-0.08) node[anchor=north west] {6};
\end{tikzpicture}
}
\caption{A graphical representation of the principle we use to distribute the vertices of $T$ for $(s,k)=(2,4)$ and in two dimensions. 
The figure shows the steps in the range $\{m',\ldots, m'+6\}$ (note that squares containing active vertices after step $m'+i$ contain only the index $i$ for reasons of space).
Note that, for $\ell = 1$, the big square with thickened black boundary is the image of the small square $p$ with thickened black boundary by $\phi_1$, and the same holds for the corresponding red squares.
Moreover, $c([0,1]^2)$, $c(p)$ and $c(\phi_1(p))$ are collinear and
$\lVert c(p) - c(\phi_1(p))\rVert\leq\lVert c([0,1]^2) - c(\phi_1(p))\rVert = \sqrt{2}(1-s^{-1})/2$.}
\label{fig 1}
\end{figure}

Now, once the sequence $(p_1, \ldots, p_t)$ is constructed, we can describe the distribution of the vertices of the tree in each step of the algorithm.
Recall that, at the beginning of the $\ell$-th block, $p_1$ contains some active vertices.
Then, for the next $t-1$ steps, we simply embed the children of all currently active vertices in $p_i$ into $p_{i+1}$ arbitrarily (which is possible thanks to~\eqref{equa:nbvertices} and~\eqref{equa:treebound3}). 
Lastly, in one more step of the algorithm, we distribute the children of all currently active vertices in $p_t$ equally among the $s^d$ subcubes of $\sigma_k(\phi_{\ell}(p))$; note that this is possible since $s^{kd}$ divides $\prod_{i=1}^{m'}s_i$ (by the definition of $s$ and since $k\leq k_2$), which itself divides the number of vertices in each of the layers $V_i$ for $i\geq m'$.
A schematic representation of the first steps of this process is shown in \Cref{fig 1}.
Finally, if $\ell=k-1$, we terminate the subroutine, and otherwise we increment the value of $\ell$ by $1$ and proceed to the next block of the second subroutine.

\vspace{1em}

Following the description of the algorithm, we must prove that it reaches a desired configuration in a suitable number of steps.
We begin by proving a bound on the number of steps of the process we have described, without regard to whether it can actually be carried out.

\begin{claim}\label{lem:nb of steps}
The second subroutine runs for at most\/ $(1-{\eps}/4)h$ steps.
\end{claim}

\begin{claimproof}
For every $\ell\in \{1,\ldots,k-1\}$, we have that 
\begin{equation}\label{eq:diagonal}
\lVert c(p) - c(\phi_{\ell}(p))\rVert\leq\lVert c(q) - c(\phi_{\ell}(p))\rVert\leq\frac{\sqrt{d}(s^{1-\ell} - s^{-\ell})}{2}, 
\end{equation}
where the first inequality holds since $c(p)$ belongs to the segment $c(q)c(\phi_{\ell}(p))$ (recall that $\phi_{\ell}(p)$ is obtained from $p$ by homothety with center $c(q)$ and ratio $s^{k-\ell} > 1$), and the second inequality comes from the fact that the farthest cubes from $c(q)$ in $\sigma_{\ell}(q)$ are the ones containing a corner of $q$.
Hence, by \eqref{eq:step1}, the total number of steps performed by the second subroutine is at most
\[\sum_{\ell=1}^{k-1} \left(\frac{\sqrt{d} (s^{1-\ell} - s^{-\ell})/2}{(1+5\eps/8) r^*}+1\right) \leq k + \frac{\sqrt{d}/2}{(1+5\eps/8)r^*}\leq k+\left(1-\frac{\eps}{3}\right)h\leq\left(1-\frac{\eps}{4}\right)h,\]
where in the second inequality we used that $\eps < 1$, and in the last we applied \eqref{equa:treebound7}.
\end{claimproof}

Now, recall that $m'+m$ denotes the index of the last step carried out by the second subroutine. 
As mentioned before describing the precise algorithm, \eqref{equa:treebound6}, \eqref{equa:nbvertices} and the condition on the point distribution in the statement guarantee that the process must succeed.
That is, we have an embedding of the layers $V_0,\ldots,V_{m'+m}$ of $T$ into $G$ and, moreover, each $q\in\cS_k$ contains the same number of vertices of $V_{m'+m}$.

Next, we make sure that the remaining layers can also be embedded into $G$.
For each cube $q\in \cS_k$, let us denote by $\hat q$ the \emph{enlarged copy of $q$} obtained from $q$ by intersecting the image of $q$ under homothety with center $c(q)$ and ratio $3$ with $[0,1]^d$ (that is, $\hat q$ is the union of all the cubes in $\cS_k$ that share at least one corner with $q$).
We show that all remaining layers of the tree can be embedded so that the following property holds: for every vertex $v \in V(G)$ that is active at the end of the second subroutine, if $v$ lies in a cube $q\in\cS_k$, then all descendants of $v$ lie in $\hat{q}$.
Observe that, by the definition of $\hat q$ and \eqref{equa:treebound2}, for every $q\in \cS_k$ we have that
\[\mathrm{diam}(\hat q) \le 3\sqrt{d} s^{-k}\le \eps r^*\le r,\]
so the vertices which lie in the enlarged copy of each cube in $\cS_k$ span a clique in $G$.
Therefore, it suffices to prove that all unseen vertices can be partitioned into $|V_{m'+m}|$ sets of equal sizes in such a way that each of the resulting sets is contained in a single $\hat{q}$ and can be assigned to one of the vertices in $V_{m'+m}$.

We achieve the above by considering an auxiliary bipartite graph $\Gamma$ with vertex partition $(\cS_k,\widehat V)$, where $\widehat V$ is the set of unseen vertices at the end of step $m'+m$.
We connect $q\in\cS_k$ and $v\in \widehat V$ by an edge whenever $v\in\hat{q}$.
Note that $|\widehat V|=\sum_{i=m'+m+1}^h \prod_{j=1}^i s_j$, so in particular $|\cS_k| = s^{kd}$ divides $|\widehat V|$. 
For a positive integer $a$, recall the definition of an $a$-star with center $v$, or with center in a set $V$, from \cref{sec:prelims}.

\begin{claim}\label{lem:last generations}
Let\/ $a\coloneqq|\widehat V|/|\cS_k|$.
Then, the vertices of\/ $\Gamma$ may be partitioned into\/ $s^{kd}$ vertex-disjoint\/ $a$-stars with centers in\/ $\cS_k$.
\end{claim}

\begin{claimproof}
Let $n \coloneqq |T|$ and note that, for each non-empty set of cubes $S\subseteq \cS_k$ with $S\neq \cS_k$, the set $\cA_S \coloneqq \bigcup_{q\in S} \hat q$ contains at least one cube in $\cS_k\setminus S$. 
Therefore, by the point-distribution property from the statement, \eqref{equa:treebound6} and \eqref{equa:nbvertices}, we have that
\COMMENT{By analyzing the last terms, we have that $|S| s^{-kd} n + s^{-kd} n - s^{kd} n^{2/3} - 2^{-(h-m-m')} n\geq |S| s^{-kd} n\geq|S|s^{-kd}|\hat V|$.
For the first inequality here, it suffices to check that $s^{-kd} n \geq s^{kd} n^{2/3} + 2^{-(h-m-m')}n$ or, equivalently, $n \geq s^{2kd} n^{2/3} + s^{kd}2^{-(h-m-m')}n$.
Since $2^{-(h-m-m')}\leq2^{-\epsilon h/4}$ by \eqref{equa:treebound4}, this inequality holds directly by \eqref{equa:treebound5} and \eqref{equa:treebound6}.} 
\begin{align*}
|\cA_S\cap \widehat V|\ge (|S|+1) (s^{-kd} n - n^{2/3}) - |V(G)\setminus \widehat V|\ge |S| s^{-kd} n + s^{-kd} n - s^{kd} n^{2/3} - s^{-kd}n/2 \geq a|S|.
\end{align*}
Thus, Corollary~\ref{cor:hall} implies that $\Gamma$ contains the desired family of disjoint $a$-stars.
\end{claimproof}

As a direct consequence of \cref{lem:last generations}, one can partition the vertices in $\widehat V$ into $s^{kd}$ disjoint sets, each spanning a clique in $G$ and contained in a different enlarged cube.
Now, each of these sets can be arbitrarily partitioned into $|V_{m'+m}|/s^{kd}$ subsets of equal sizes, and each of these smaller sets can be assigned to a distinct vertex $v$ in $V_{m'+m}$.
Then, for every $v\in V_{m'+m}$, the descendants of $v$ can be embedded greedily into the complete graph associated to it,
and the embedding of $T$ into $G$ is completed.
\end{proof}

\begin{remark}
    By retracing the proof of \cref{thm:deterministic}, and taking into account \cref{rem:polynomial}, one can readily verify that our approach results in a polynomial time algorithm to find a copy of an $M$-tree in the corresponding geometric graph.
\end{remark}

With this, we can complete the proof of \cref{thm:tree}.

\begin{proof}[Proof of \cref{thm:tree}]
Suppose that $r \le (1-\eps)r^*$.
First, fix a tree $T$ of height $h = O(\log |T|)$ (which holds for every $M$-tree).
Set $G = \cG(|T|,r,d)$.
To show that $\mathbb{P}[T\subseteq G] = o(1)$, we first claim that a.a.s.\ there are vertices $u,v\in V(G)$ at Euclidean distance $(1-o(1))\sqrt{d}$ from each other. 
Note that, since the graph sequence $(G_n)_{n\ge 1}$ is increasing with respect to inclusion, it is sufficient to show the previous statement for the smallest graph hosting an $M$-tree of height $h$, that is, $G_{2^{h+1}-1}$.

To this end, consider the cubes $c_0, c_1\subseteq [0,1]^d$ of side length $1/h$ containing the corners $\{0\}^d$ and $\{1\}^d$, respectively.
Then, we have that $\dist(c_0,c_1)=(1-o(1))\sqrt{d}$ and, moreover, $|c_0\cap V(G_{2^{h+1}-1})|$ and $|c_1\cap V(G_{2^{h+1}-1})|$ both follow a binomial distribution with parameters $2^{h+1}-1$ and $1/h^d$.
Thus, the probability that $c_0$ or $c_1$ do not contain any vertices is at most
\[2(1-1/h^d)^{2^{h+1}-1}\leq 2\nume^{-2^h/h^d}.\]
Now, condition on the event that $c_0$ and $c_1$ each contain at least one vertex and suppose that $G$ admits $T$ as a spanning tree.
Since $T$ has diameter $2h$, $u$ and $v$ must be at distance at most $2h$ in $G$.
Thus, by the triangle inequality, the Euclidean distance between them must be at most $2hr$, so we must have $2hr \ge \dist(c_0, c_1) = (1-o(1))\sqrt{d}$, which is a contradiction with our choice of $r$.
Hence, with probability at least $1-2\nume^{-2^h/h^d}$, there is no $M$-tree $T$ of height $h$ such that $T\subseteq\cG(|T|,r,d)$.
This concludes the proof of \ref{thm:item1}.

% Moreover, the number of $M$-trees of height $h$ is at most $M^h$ (for each layer, one needs to choose the number of children of all vertices in that layer, which is between $2$ and $M$), and so, using that $|T| \ge 2^h$, with probability at least 
% \[1- 2M^h\nume^{-|T|/h^d} \ge 1-2M^h \nume^{-2^h/h^d}=1-o(1),\]
% there is no $M$-tree $T$ of height $h$ such that $T\subseteq\cG(|T|,r,d)$.
% This concludes the proof of \ref{thm:item1}.

Next, suppose that $r\ge (1+\eps)r^*$.
Using that $T\subseteq G$ is an increasing property, it suffices to prove the result for $r = (1+\eps) r^*$.
In fact, it suffices to verify that a.a.s.\ the conditions of the statement of \cref{thm:deterministic} hold for every sufficiently large graph in the random geometric graph sequence.
The next claim asserts that the point-distribution property from \cref{thm:deterministic} holds a.a.s.

\begin{claim}\label{claim:tree1}
A.a.s.\ for every\/ $s\in \{2,\ldots,M\}$, every cube\/ $q\in\cW_s$ and every\/ $n\ge 2^h$,\/ $q$ contains\/ $s^{-k_s d} n \pm n^{2/3}$ vertices of\/ $\cG(n,r,d)$.
\end{claim}

\begin{claimproof} 
Fix $s$ and $n$ as above and consider a fixed cube $q\in\cW_s$.
Let $X$ denote the number of vertices of $\cG(n,r,d)$ inside $q$.
Then, $X$ is a binomial random variable with parameters $n$ and $p=s^{-k_s d}$.
Therefore, $\mathbb{E}[X]=s^{-k_s d}n$ and, by \cref{lem:Chernoff},
\begin{align*}
\mathbb{P}[X\neq s^{-k_s d}n\pm n^{2/3}] 
&= \mathbb{P}\left[|X-\mathbb{E}[X]|>\frac{n^{2/3}}{s^{-k_s d}n}s^{-k_s d}n\right]\\
&\leq 2\exp\left(-\left(\frac{n^{2/3}}{s^{-k_s d}n}\right)^2\frac{s^{-k_s d}n}{3}\right)\le 2\exp\left(-\frac{n^{1/3}}{3}\right).
\end{align*}
A union bound over the $M-1 = o(h)$ possible values of $s$, the $s^{k_s d} = O(r^{-2}) = O(h^{2d})$ cubes in $\cW_s$ and all $n\ge 2^h$ shows that the event of the lemma holds with probability at least
\[1 - O\left(\sum_{n = 2^h}^{\infty} 2h^{2d+1} \exp\left(-\frac{n^{1/3}}{3}\right)\right) = 1 - o(1),\] 
as desired.
\end{claimproof}

Since with probability $1$ none of the vertices of any of the graphs in the random geometric graph sequence is on the boundary of any cube $q\in\mathcal{W}_{s}$, the proof of \cref{thm:tree} is completed.
\end{proof}

\section{Extensions and concluding remarks}

\subsection{Other metric spaces}

As a first remark, we note that all our results can be extended to other $\ell_p$ norms where $1\leq p\leq\infty$.
Indeed, the notion of random geometric graph can be adapted to each norm simply by replacing the Euclidean distance in the definition with the distance in the $\ell_p$ norm.
Then, by retracing the proof of our main theorem, the sharp threshold becomes $d^{1/p}/2h$.

\subsection{Threshold width}

As mentioned in the introduction, \citet{GRK05} gave upper bounds on the threshold width for any monotone increasing property in $\cG(n,r,d)$.
Let us state their result.
For any positive integer $n$, real number $x\in[0,1]$, and increasing property $\mathcal{P}$, let
\[
r_{\mathcal{P}}(n,x)\coloneqq\inf \{r \ge 0: \Pr[\cG(n,r,d) \in \mathcal{P}] \geq x\}.
\]
For each $\epsilon\in(0,1/2)$, define the \emph{$\epsilon$-threshold width} of property $\mathcal{P}$ as $\delta_{\mathcal{P}}(n,\epsilon)\coloneqq r_{\mathcal{P}}(n,1-\epsilon)-r_{\mathcal{P}}(n,\epsilon)$.
\citet{GRK05} showed that, for any monotone increasing property $\mathcal{P}$, for $d=1$ we have $\delta_{\mathcal{P}}(n,\epsilon)=O(\sqrt{\log \epsilon^{-1}/n})$ (in fact, this bound is also sharp), for $d=2$ we have $\delta_{\mathcal{P}}(n,\epsilon)=O(\log^{3/4}n/\sqrt{n})$, and for $d \ge 3$ we have $\delta_{\mathcal{P}}(n,\epsilon)=O(\log^{1/d} n / n^{1/d})$.
(Some references, such as the paper by Goel, Rai and Krishnamachari~\cite{GRK05}, refer to the fact that $\delta_{\mathcal{P}}(n,\epsilon)=o(1)$ as property $\mathcal{P}$ having a sharp threshold.
We note, however, that this is different from the notion of a sharp threshold we have considered throughout.)

The threshold width for the property of containing a complete balanced $s$-ary tree that can be derived from our methods is far from the upper bound of \citet{GRK05}.
In fact, by retracing our proof, it only follows that the width is $O(\log\log h/(h\log h))$.
Since the bounds given by Goel, Rai and Krishnamachari are tighter, it is natural to wonder about the exact width of the threshold window.
The following problem addresses this question: 

\begin{problem}
Determine the threshold width for the property that $\cG(n,r,d)$ contains a spanning balanced $s$-ary tree.
\end{problem}

\COMMENT{The ultimate problem would be to determine the limiting distribution of the probability that the random graph contains any given balanced tree.}

\subsection{Balanced \texorpdfstring{$s$}{s}-ary trees of all orders}\label{sect:prebalanced}

In the introduction we defined a (complete) balanced $s$-ary tree to be a tree with layers $V_0,\ldots,V_h$ where all vertices in layers $V_0,\ldots,V_{h-1}$ have exactly $s$ children.
One weakness of this definition is the fact that these trees are only defined for certain orders: indeed, any such tree must have $\sum_{i=0}^hs^i$ vertices.
One may consider a more general definition which works for all possible orders as follows.
Given any positive integers $n$ and $s$, let $h=h(n)$ be the unique positive integer such that 
\begin{equation}\label{equa:generaldef}
    \sum_{i=0}^{h-1}s^i<n\leq\sum_{i=0}^hs^i.
\end{equation}
A \emph{pre-balanced $s$-ary tree} on $n$ vertices is then a tree with layers $V_0,\ldots,V_h$ such that all vertices in layers $V_0,\ldots,V_{h-2}$ have exactly $s$ children and every vertex in layer $V_{h-1}$ has at most $s$ children.
Note that, when \eqref{equa:generaldef} is not satisfied with equality in the upper bound, there may be multiple non-isomorphic pre-balanced $s$-ary trees.

One would expect that our results extend to this more general definition of pre-balanced trees.
This, however, is not immediate.
Intuitively, the reason for this is that pre-balanced $s$-ary trees can be quite ``unbalanced'' in the following sense: For any $s$-ary tree, let us define the \emph{weight} of a vertex as the number of descendants it has. 
Then, in complete balanced $s$-ary trees, all vertices in the same layer have the same weight.
In pre-balanced $s$-ary trees, however, the weights of vertices in the same layer may differ by a factor of~$s$.
(Indeed, consider the pre-balanced $s$-ary tree obtained by taking $s$ complete balanced $s$-ary trees, with $s-1$ of height $h-1$ and one of height $h-2$, and having the roots of each of these be the children of a new root).

Our proof extends directly to pre-balanced $s$-ary trees of height $h$ where, for some integer $m'' = m''(h) = h-o(h)$, all vertices in $V_{m''}$ have the same weight.
On the other extreme, consider a pre-balanced $s$-ary tree that has vertices of arbitrarily different weights in the same layer but satisfies the following property for some integer $\widehat{m}=\widehat{m}(h)=o(h)$: for every vertex $v\in V_{\widehat{m}}$, if we consider the tree induced by $v$ together with all of its descendants, then all vertices in the same layer of this subtree have the same weight (this case contains the example of $s$ complete balanced trees with an added common root mentioned above).
In this case, we can greedily embed all vertices in layers $V_0,\ldots,V_{\widehat{m}}$ into a square in $\sigma_k([0,1]^d)$, and then apply the algorithm to each of the $|V_{\widehat{m}}|$ trees rooted at the vertices in layer $V_{\widehat{m}}$ simultaneously.
However, if the conditions described above are not met (e.g., if all vertices in the first $\epsilon h$ layers have the same weight but then weights become very different), it is unclear how to extend our proof.

Despite this, we still believe that the thresholds should be the same, and even that all pre-balanced $s$-ary trees of height $h$ on $n$ vertices should appear essentially simultaneously.
We therefore propose the following universality conjecture: 

\begin{conjecture}
    Fix positive integers\/ $s\geq2$ and\/ $d\geq1$.
    For a positive integer\/ $n$, set\/ $h$ to be the unique integer satisfying \eqref{equa:generaldef}.
    Then,\/ $r^*\coloneqq\sqrt{d}/2h$ is the sharp threshold for the event that\/ $\cG(n,r,d)$ contains a copy of every pre-balanced\/ $s$-ary tree of height\/ $h$ on\/ $n$ vertices.
\end{conjecture}

We note that the discussion in this section can also be extended to balanced trees over a sequence $(s_i)_{i\geq1}$, where the extension of the definition is analogous to the extension for $s$-ary trees.

\subsection{Bounded-degree trees}

As mentioned in the introduction, a class of trees of particular interest is the class of bounded-degree trees.
This class contains some of the trees we have considered in this paper.
For instance, let $T^s_h$ be the balanced tree of height $h$ over the sequence $(s_i)_{i=1}^h$ where $s_1=s$ and $s_i=s-1$ for all $i\in\{2,\ldots,h\}$.
The following is an immediate consequence of \cref{thm:treev2}.

\begin{corollary}
Let\/ $d$ and\/ $s\geq3$ be fixed integers.
Let\/ $h$ be an integer and\/ $n$ be the number of vertices of\/ $T_h^s$.
The sharp threshold for the appearance of\/ $T_h^s$ as a subgraph of\/ $\cG(n,r,d)$ is\COMMENT{The number of vertices $n$ of the $s$-regular tree of height $h$ is
\[n=1+s\sum_{i=0}^{h-2}(s-1)^i=1+s\frac{(s-1)^{h-1}-1}{s-2}.\]
Working to isolate $h$, we have that
\[n=1+s\frac{(s-1)^{h-1}-1}{s-2}\iff\frac{s-2}{s}(n-1)+1=(s-1)^{h-1}\iff h=\log_{s-1}\left(\frac{s-2}{s}(n-1)+1\right)+1.\]
Now, in the threshold we only care about asymptotic terms, so (assuming $s$ is constant) we have 
\[h=\log_{s-1}\left(\frac{s-2}{s}(n-1)+1\right)+1\sim\log_{s-1}\left(\frac{s-2}{s}(n-1)\right)\sim\log_{s-1}\left(\frac{s-2}{s}n\right)=\frac{\log\left(\frac{s-2}{s}n\right)}{\log(s-1)}=\frac{\log\frac{s-2}{s}}{\log(s-1)}+\frac{\log n}{\log(s-1)}\sim\frac{\log n}{\log(s-1)}.\]
Now we simply substitute this into the formula of the threshold to obtain the result.}
\[\frac{\sqrt{d}}{2h}\sim\frac{\sqrt{d}\log(s-1)}{2\log n}.\]
\end{corollary}

We suspect that, for each fixed value of $s$, among all trees with degrees bounded by $s$, the tree $T_h^s$ is the hardest to embed in the sense that it has the highest threshold. 
Thus, we propose the following universality conjecture:

\begin{conjecture}
Fix positive integers\/ $d$ and\/ $s\geq3$.
For any\/ $\epsilon>0$, if 
\[r\geq(1+\epsilon)\frac{\sqrt{d}\log(s-1)}{2\log n},\] 
then a.a.s.\ $\cG(n,r,d)$ contains \emph{every} spanning tree of maximum degree at most\/ $s$.
\end{conjecture}

As discussed in the introduction, though, we know that there exist trees of bounded degree whose thresholds are far smaller than those discussed here.
It would be interesting to understand which properties of a tree can be used to determine their threshold.
We note that, while the diameter has resulted to be a good estimator throughout this paper, this is not the correct parameter in general.
Indeed, consider a tree on $2n$ vertices defined as follows: take a balanced binary tree on $n$ vertices, and add a path on $n$ vertices appended to the root.
The resulting tree has linear diameter, but it is not hard to see that the threshold for its appearance is $\Omega(1/\log n)$. 
As a step towards the general problem of determining the threshold of each family of bounded-degree trees, we propose the following problem:

\begin{problem}
    Given any tree\/ $T$, let \/$\mathrm{diam}(T)$ denote the diameter of\/ $T$.
    Give general sufficient conditions which guarantee that, for a sequence\/ $(T_i)_{i\geq1}$ of trees, the sharp threshold for their appearance in\/ $\cG(n,r,d)$ is\/ $\sqrt{d}/\mathrm{diam}(T_i)$.
\end{problem}

% Use with natbib, not biblatex:
\bibliographystyle{mystyle} 
\bibliography{TreeRGG}

% Use with biblatex, not natbib:
% \printbibliography

\end{document}